\documentclass[a4paper,
               11pt,
               pdftex,
               headsepline,
               footsepline,
               onecolumn,
               headinclude,
               footinclude,
               abstracton
               ]{scrartcl}

\usepackage
{
    graphicx,
    amssymb,
    amsmath,
    amsthm,
    xcolor,
    dsfont,
    algpseudocode,
    authblk,
    enumitem
}

\usepackage[bf,normal]{caption}

\usepackage[colorlinks,
            pdffitwindow=false,
            plainpages=false,
            pdfpagelabels=true,
            pdfpagemode=UseOutlines,
            pdfpagelayout=SinglePage,
            bookmarks=false,
            colorlinks=true,
            hyperfootnotes=false,
            linkcolor=blue,
            citecolor=green!50!black]{hyperref}

\usepackage{tikz} 
\usepackage{subcaption} 

\DeclareMathAlphabet{\mathpzc}{OT1}{pzc}{m}{it}

\newcommand{\subfiguretitle}[1]{{\scriptsize{#1}} \\[1mm] }
\newcommand{\R}{\mathbb{R}}
\newcommand{\C}{\mathbb{C}}

\providecommand{\norm}[1]{\left\lVert #1 \right\rVert}

\DeclareMathOperator{\tr}{tr}

\DeclareMathOperator{\Aut}{Aut}

\newtheorem{theorem}{Theorem}[section]
\newtheorem{corollary}[theorem]{Corollary}

\newtheorem{proposition}[theorem]{Proposition}
\newtheorem{definition}[theorem]{Definition}
\theoremstyle{definition}
\newtheorem{example}[theorem]{Example}
\newtheorem{remark}[theorem]{Remark}

\makeatletter
\renewcommand*\env@matrix[1][*\c@MaxMatrixCols c]{%
  \hskip -\arraycolsep
  \let\@ifnextchar\new@ifnextchar
  \array{#1}}
\makeatother

\begin{document}

\title{On the Equivariance Properties \\ of Self-adjoint Matrices}
\author[*]{Michael Dellnitz, Bennet Gebken, Raphael Gerlach}
\author[**]{\\Stefan Klus}
\affil[*]{\footnotesize Department of Mathematics, Paderborn University, D-33095 Paderborn, Germany}
\affil[**]{\footnotesize Department of Mathematics and Computer Science, Freie Universit\"at Berlin, D-14195 Berlin, Germany}
\maketitle

\begin{abstract}
	We investigate self-adjoint matrices $A\in\R^{n,n}$ with respect to their equivariance properties. We show in particular that a matrix is self-adjoint if and only if it is equivariant with respect to the action of a group $\Gamma_2(A)\subset \mathbf{O}(n)$ which is isomorphic to $\otimes_{k=1}^n\mathbf{Z}_2$. If the self-adjoint matrix possesses multiple eigenvalues -- this may, for instance, be induced by symmetry properties of an underlying dynamical system -- then $A$ is even equivariant with respect to  the action of a group $\Gamma(A) \simeq \prod_{i = 1}^k \mathbf{O}(m_i)$ where
	$m_1,\ldots,m_k$ are the multiplicities of the eigenvalues $\lambda_1,\ldots,\lambda_k$ of $A$. We discuss implications of this result for equivariant bifurcation problems, and we briefly address further applications for the Procrustes problem, graph symmetries and Taylor expansions.
\end{abstract}



%

\section{Introduction}

For more than 30 years the influence of symmetry properties of a dynamical
system on its qualitative temporal behavior has been intensively studied. Such symmetry
properties are typically induced by network structures or geometric properties of the underlying
mathematical model. The related research focuses on a variety of topics, for instance, the
classification of symmetry breaking bifurcations (\cite{GSS88}) or the explanation of the
occurrence of stable heteroclinic cycles. For an overview of this area and their relevance
in the sciences we refer to \cite{GS03}.

Formally, symmetry properties of a dynamical system $\dot x = f(x)$
manifest themselves by an equivariance property of the right-hand side.
That is, $f\colon\R^n \to \R^n$ satisfies
\begin{equation*}
	\gamma f(x) = f(\gamma x)\quad \mbox{for all $\gamma \in \Gamma$,}
\end{equation*}
where $\Gamma \subset \mathbf{O}(n)$ is a compact Lie group. It is
well known that equivariance properties are inherited by the
linearization $Df(x^*)$ of $f$ from the symmetry properties of the
steady-state solutions $x^*$.
In fact, this is the reason why generically $Df(x^*)$ may possess multiple
eigenvalues, which implies the occurrence of complex symmetry breaking
bifurcations in dynamical systems. This happens, for instance, if $\Gamma = \mathbf{O}(n)$
($n\ge 2$) or $\Gamma = \mathbf{D}_\ell$ ($\ell\ge 3$), where $\mathbf{D}_\ell$ is the dihedral group
of order $\ell$, that is, the symmetry group of the $\ell$-sided regular polygon.

The investigations in this article are motivated by the analysis of
equivariant dynamical systems where the linearization $A = Df(x)$ is additionally
self-adjoint, that is, the matrix $A\in \R^{n,n}$ satisfies $A = A^T$.
Recently, it has been observed that a matrix is
self-adjoint if and only if it is equivariant with respect to 
the action of a group $\Gamma_2(A)\subset \mathbf{O}(n)$
which is isomorphic to $\prod_{i = 1}^n\mathbf{Z}_2$ (see \cite{De17}).
This underlying equivariance property
is implicitly present in articles concerning the development of dynamical
systems for the solution of certain optimization problems
(e.g.\ \cite{Scho68, Bro89, Bro91}). But to the best of our knowledge
it has not explicitly been stated elsewhere before -- and definitely not in
the dynamical systems context.

In this article, we extend this result from \cite{De17} significantly
in the sense that we completely
characterize the equivariance properties of self-adjoint matrices by their
spectra. In fact, we will show in our main result on the equivariance properties
of self-adjoint matrices (Corollary~\ref{cor:sadj}) that
$\Gamma(A)$ is isomorphic to $\prod_{i = 1}^k \mathbf{O}(m_i)$ where
$m_1,\ldots,m_k$ are the multiplicities of the eigenvalues
$\lambda_1,\ldots,\lambda_k$ of $A$. In particular, if $A$ has only
simple eigenvalues, then $\Gamma(A) = \Gamma_2(A) \simeq \prod_{i = 1}^n\mathbf{Z}_2$.

One important consequence of this result is the following observation: 
Suppose that the underlying dynamical system is $\mathbf{D}_\ell$-equivariant
for an $\ell \ge 3$. Then -- as already mentioned above -- the linearization
$Df(x^*)$ at a $\mathbf{D}_\ell$-symmetric steady-state solution $x^*$
generically possesses double eigenvalues. Our result implies that in this
case the linearization will not just be $\mathbf{D}_\ell$-equivariant but (at least) even
equivariant with respect to an action of
$\Gamma(A) \simeq \mathbf{O}(2) \times \prod_{i = 1}^{n-2}\mathbf{Z}_2$.

Moreover, if in addition the entire function $f$ is $\Gamma(A)$-equivariant, then
symmetry-related bifurcations of the system will be governed by
$\Gamma(A)$ rather than $\mathbf{D}_\ell$, and this leads
to phenomena which would generically be unexpected if only $\mathbf{D}_\ell$
is taken into account. This would apply, for instance, to numerical discretizations
of the cubic or quintic Ginzburg--Landau equation on a $\mathbf{D}_\ell$-symmetric spatial domain
(\cite{CM96,AK02}).
Thus, from an abstract point of view our results are strongly related
to the notion of \emph{hidden symmetries} which has been
introduced in connection with the occurrence of unexpected bifurcations in partial differential
equations with Neumann boundary conditions (\cite{DA86,GS03}).
We will illustrate this fact by several examples in the
following sections.

A detailed outline of the structure of this article is as follows.
In Section~\ref{sec:Motivation}, we introduce a specific
$\mathbf{Z}_2$-equivariant dynamical system as a guiding example.
This system exhibits unexpected dynamical phenomena driven by the
underlying $\Gamma(A)$-equivariance: A symmetry-preserving
pitchfork bifurcation and the existence of an entire orbit of steady-state
solutions.
Then, in Section~\ref{sec:MR}, we review briefly the main result from \cite{De17}.
This will allow us to reveal the structure which leads to the symmetric pitchfork
bifurcation in the guiding example. Our main results concerning the 
equivariance properties of self-adjoint matrices are stated in Section~\ref{sec:mainres}.
In Section~\ref{sec:Impl}, we discuss the consequences for bifurcations in equivariant
dynamical systems.
Finally, in Section~\ref{sec:appl}, we discuss a couple of further applications:
First we characterize all solutions of the two-sided orthogonal Procrustes problem
(Section \ref{sec:app1}).
Then we briefly discuss consequences for the graph isomorphism problem for undirected
graphs (Section \ref{sec:app2}). We conclude with the construction of simple 
approximations of derivatives of higher order for real valued functions (Section \ref{sec:app3}).
Here we make use of the fact that each Hessian $H$ is symmetric and 
therefore also $\Gamma(H)$-equivariant.

\section{Motivation -- the Guiding Example}
\label{sec:Motivation}
As a guiding example we consider the differential equation

\begin{equation} \label{eq:DSM}
	\dot{x} = A(\mu) x - \norm{x}_2^2 x, 
\end{equation}
where $x\in \R^3,\mu \in\R$ and
\begin{equation*}
	A(\mu) = \begin{pmatrix}
		2 & \sqrt{2}(2 \mu - 1) & \sqrt{2}(2 \mu - 1) \\
		\sqrt{2}(2 \mu - 1) & 3 - 2 \mu & 2 \mu - 1 \\
		\sqrt{2}(2 \mu - 1) & 2 \mu - 1 & 3 - 2 \mu
	\end{pmatrix}.
\end{equation*}
Observe that this problem has an obvious $\mathbf{Z}_2$-symmetry: first the
matrix $A(\mu)$ commutes for all $\mu$ with the permutation matrix
\begin{equation}
	\label{eq:S}
	S = \begin{pmatrix}
		1 & 0 & 0 \\
		0 & 0 & 1 \\
		0 & 1 & 0
	\end{pmatrix}.
\end{equation}
That is,
\[
S A(\mu) =  A(\mu) S \quad \mbox{for all $\mu\in\R$.}
\]
Moreover, $\norm{x}_2^2$ is invariant under orthogonal transformations. Therefore, the
right-hand side $f(x,\mu) = A(\mu) x -\norm{x}_2^2 x$ in \eqref{eq:DSM} is $\mathbf{Z}_2$-equivariant and satisfies
\[
f(Sx,\mu) = Sf(x,\mu) \quad \mbox{for all $x\in \R^3$ and $\mu\in\R$.}
\]
Thus, by genericity results from classical bifurcation or singularity theory (see \cite{GSS88}) we would
particularly expect that 
\begin{itemize}
	\item[$(i)$] the only steady-state bifurcations that occur in \eqref{eq:DSM} are turning points or (symmetry-breaking)
	pitchfork bifurcations corresponding to the underlying symmetry given by
	$\mathbf{Z}_2 = \{ I, S\}$ (e.g.\ \cite{GSS88,WS84});
	\item[$(ii)$] equilibria of \eqref{eq:DSM} are isolated.
\end{itemize}

In contrast to this expectation we observe the following two phenomena for \eqref{eq:DSM}:
\begin{itemize}
	\item[$(i)^*$] Apparently the system undergoes a pitchfork bifurcation at $\mu=0$. The corresponding
	local bifurcation diagram is shown in Figure~\ref{fig:BF}~(a). However, at the bifurcation point a normalized kernel vector
	of $A(0)$ is given by
	\begin{equation}
		\label{eq:ev}
		v=\left(\frac{1}{\sqrt{2}},\frac{1}{2},\frac{1}{2}\right)^T.
	\end{equation}
	In particular, this eigenvector is $S$-symmetric ($S v = v$) rather than antisymmetric ($S v = -v$)
	as expected. Accordingly, also the equilibria on the bifurcating branches
	are $S$-symmetric, see Figure \ref{fig:BF} (a).
	
	\begin{figure}
		\centering
		\begin{minipage}[t]{0.49\textwidth}
			\centering
			\subfiguretitle{(a)}
			\includegraphics[width = 0.9\textwidth]{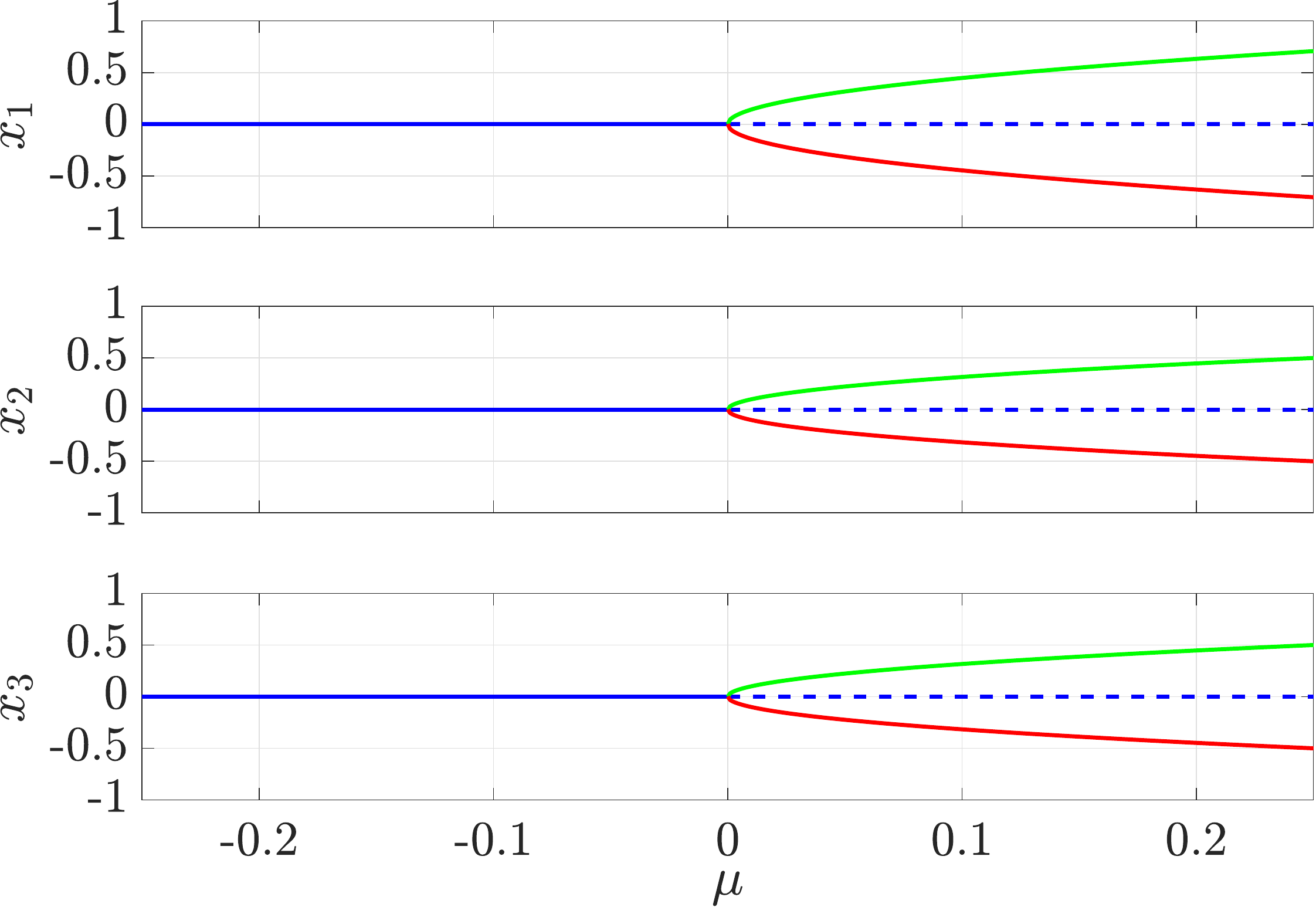}
		\end{minipage}
		\begin{minipage}[t]{0.49\textwidth}
			\centering
			\subfiguretitle{(b)}
			\includegraphics[width = 0.95\textwidth]{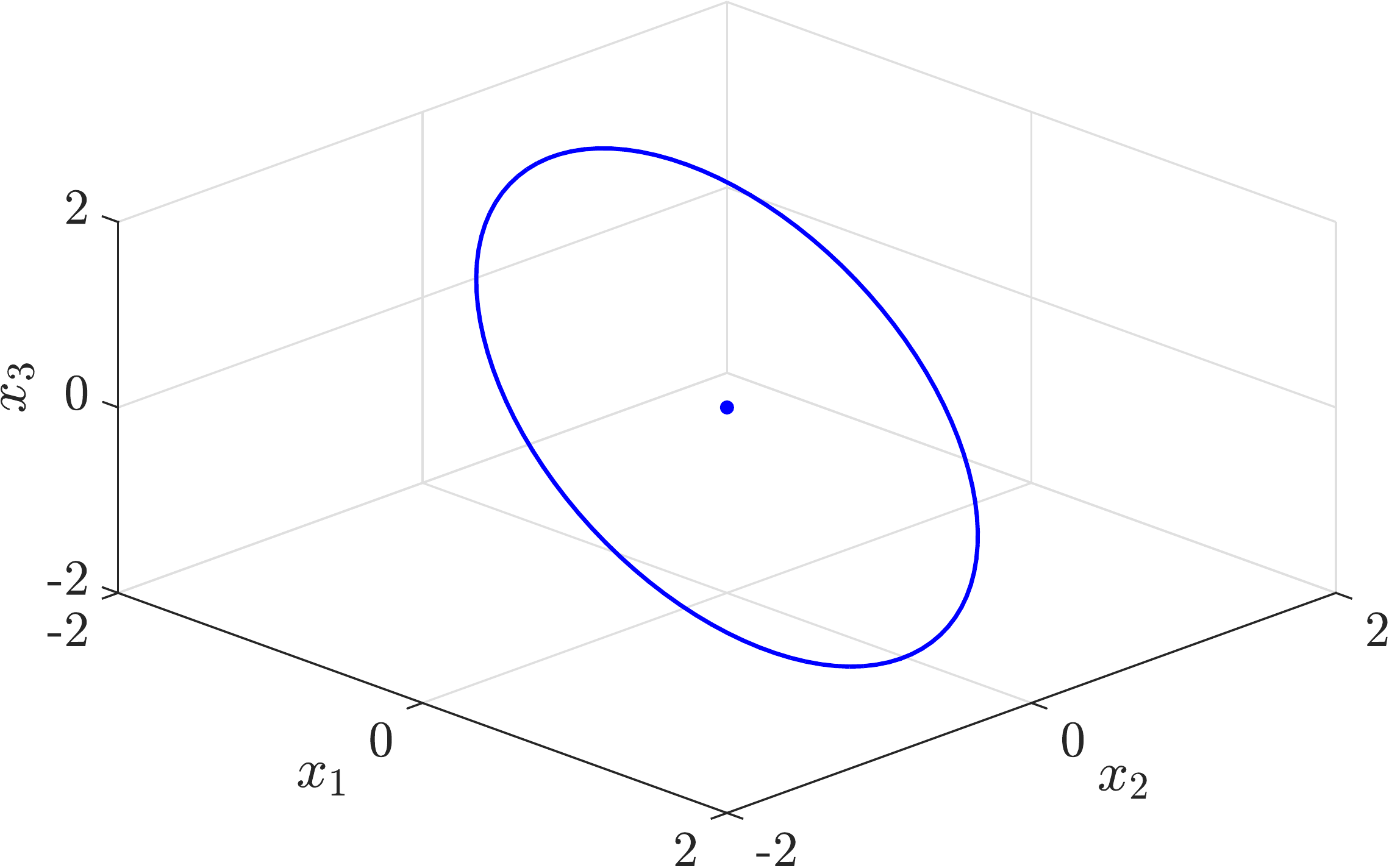}
		\end{minipage}
		\caption{(a) Illustration of the `symmetric' pitchfork bifurcation for $\mu = 0$
			in the dynamical system \eqref{eq:DSM}.
			The bifurcating steady-state solutions are $S$-symmetric, that is $x_2 = x_3$
			on the green and red branches, whereas  $x_1 \neq x_2$ and $x_1 \neq x_3$
			on these branches as expected. (b) The set of equilibria of the dynamical system \eqref{eq:DSM} for $\mu=-0.25$.}
		\label{fig:BF}
	\end{figure}
	
	\item[$(ii)^*$] For $\mu=-0.25$ we find not just $x^* = 0$, but in addition
	an entire continuous orbit of equilibria for \eqref{eq:DSM}, see Figure~\ref{fig:BF} (b). In fact, we will see in Section \ref{sec:Impl} that such orbits exist for an entire range of parameter values.
\end{itemize}

It is the purpose of this work to explain such phenomena, and we will
see that this is strongly related to the fact that $A$ is self-adjoint.
In fact, our results will imply that the dynamical system \eqref{eq:DSM} is
$\mathbf{Z}_2 \times \mathbf{O}(2)$-equivariant with $\{I,S\} \subset \mathbf{O}(2)$,
and this will explain the phenomena described in $(i)^*$ and $(ii)^*$.

\section{Self-adjoint Matrices are Equivariant -- a Warm-up}
\label{sec:MR}

In this section, we briefly summarize the main result from \cite{De17}.
With this we illustrate the underlying structure, namely that 
equivariance properties of self-adjoint matrices are induced by the 
symmetry properties of diagonal matrices.

Let $\Sigma \subset \mathbf{O}(n)$ be the abelian group consisting of the
$2^n$ matrices
\[
\begin{pmatrix}
\pm 1   & 0 & \cdots   & 0 \\
0 & \pm 1   & \ddots & \vdots   \\
\vdots  & \ddots & \ddots & 0 \\
0 & \cdots   & 0 & \pm 1
\end{pmatrix}.
\]
For any diagonal matrix
\[
D = 
\begin{pmatrix}
\lambda_1   & 0 & \cdots   & 0 \\
0 & \lambda_2   & \ddots & \vdots   \\
\vdots  & \ddots & \ddots & 0 \\
0 & \cdots   & 0 & \lambda_n
\end{pmatrix}, \quad \lambda_i \in \R, \quad i=1,2,\ldots,n,
\]
we obviously have
\begin{equation*}
	\sigma D = D \sigma \quad \forall \sigma \in \Sigma.
\end{equation*}

In fact, it is easy to verify for an arbitrary matrix $B\in\R^{n,n}$ that

\begin{equation}
	\label{eq:DSigma}
	\sigma B = B \sigma \quad \forall \sigma \in \Sigma \quad \Longleftrightarrow \quad
	\mbox{$B$ is a diagonal matrix.}
\end{equation}

\begin{proposition}[\cite{De17}]
	\label{prop:main}
	A matrix $A\in \R^{n,n}$ is self-adjoint (i.e.\ $A = A^T$) if and only if there is
	an orthogonal matrix $V\in \mathbf{O}(n)$ such that
	\begin{equation}
		\label{eq:equiv}
		\gamma A = A \gamma\quad \forall \gamma \in \Gamma_2(A),
	\end{equation}
	where the group $\Gamma_2(A) \subset \mathbf{O}(n)$ is defined by
	\begin{equation}
		\label{eq:GV}
		\Gamma_2(A) = \{ V^T \sigma V : \sigma \in \Sigma \}.
	\end{equation}
\end{proposition}

We state the proof for the sake of completeness.
\begin{proof}
	Suppose that $A = A^T$. Then there is
	$V\in \mathbf{O}(n)$ such that
	\[
	D = V A V^T
	\]
	is a diagonal matrix.
	By \eqref{eq:DSigma}
	we have for all $\sigma \in \Sigma$
	\begin{equation*}
		\sigma VAV^T = VAV^T \sigma \quad \Longleftrightarrow
		\quad V ^T\sigma VA = AV^T \sigma V.
	\end{equation*}
	Therefore, $A$ satisfies the equivariance condition \eqref{eq:equiv}.
	
	Now suppose that  \eqref{eq:equiv} is satisfied for some $V\in \mathbf{O}(n)$.
	Then the matrix $V A V^T$ commutes with every $\sigma\in\Sigma$, and 
	by \eqref{eq:DSigma}
	it follows that $D=VAV^T$ is a diagonal matrix. Therefore,
	\[
	A^T = (V^T D V)^T = A
	\]
	as desired.
\end{proof}

\begin{remark} \label{rmk:main} ~
	\begin{itemize}
		\item[(a)] Observe that the implication `$\Longrightarrow$' could also be proved by
		using the well-known fact that two matrices $A$ and $B$ commute if 
		there is an orthogonal transformation $V$ such that both $V^T A V$ and
		$V^T B V$ are diagonal.
		\item[(b)] By construction all the eigenvalues of every $\gamma \in \Gamma_2(A)$ are $1$ or $-1$.
		In particular $\gamma ^2 = I$ for all $\gamma \in \Gamma_2(A)$, and $\Gamma_2(A) \simeq \prod_{i = 1}^n\mathbf{Z}_2$.
		Moreover, by (a) the matrix $A$ and all $\gamma\in \Gamma_2(A)$ possess the same set of eigenvectors.
		\item[(c)] Obviously, analogous results can be obtained for Hermitian or normal matrices:
		Using essentially the same proof as in Proposition~\ref{prop:main} one can show that
		a matrix $A\in \C^{n,n}$ is normal (i.e.\ $AA^* = A^* A$) if and only if there is
		a unitary matrix $W\in \mathbf{U}(n)$ such that
		\[
		\gamma A = A \gamma\quad \forall \gamma \in \Gamma_2(A),
		\]
		where the group $\Gamma_2(A) \subset \mathbf{U}(n)$ is defined by
		\[
		\Gamma_2(A) = \{ W^* \sigma W : \sigma \in \Sigma \}.
		\]
	\end{itemize}
\end{remark}

\begin{example}
	\label{ex:pitchfork}
	Let us consider the matrix $A(0)$ from our guiding example in Section~\ref{sec:Motivation}, i.e.
	\begin{equation*}
		A(0) = \begin{pmatrix}
			2 & -\sqrt{2} & -\sqrt{2}\\
			-\sqrt{2} & 3 & -1 \\
			-\sqrt{2} & -1 & 3
		\end{pmatrix}.
	\end{equation*}
	The matrix
	\[
	V =
	\begin{pmatrix}
	-0.7071 & -0.5  &  -0.5 \\
	0.6969  & -0.3732  &  -0.6124\\
	0.1196  &  -0.7815 &  0.6124
	\end{pmatrix}
	\]
	transforms $A(0)$ into a diagonal matrix $D$ with the eigenvalues $0,4,4$
	of $A(0)$ on the 
	diagonal.
	With Proposition~\ref{prop:main} we can compute eight matrices
	which commute with $A(0)$ -- the elements of $\Gamma_2(A)$ --, and these matrices are given by
	\begin{align*}
		&\gamma_1=\begin{pmatrix}
			1 & 0 & 0\\
			0 & 1 & 0\\
			0 & 0 & 1
		\end{pmatrix}, &
		\gamma_2=\begin{pmatrix}
			-0.9714 & -0.1869 & 0.1464\\
			-0.1869 & 0.2214 & -0.9571\\
			0.1464 & -0.9571 & -0.2500
		\end{pmatrix},\\
		&\gamma_3=\begin{pmatrix}
			0 & \frac{1}{\sqrt{2}} & \frac{1}{\sqrt{2}}\\
			\frac{1}{\sqrt{2}} & -\frac{1}{2} & \frac{1}{2}\\
			\frac{1}{\sqrt{2}} & \frac{1}{2} & -\frac{1}{2}
		\end{pmatrix},
		&\gamma_4=\begin{pmatrix}
			0.0286 & 0.5202 & 0.8536\\
			0.5202 & 0.7214 & -0.4571\\
			0.8536 & -0.4571 & 0.2500
		\end{pmatrix}
	\end{align*}
	and $\gamma_{j+4}=-\gamma_j$ for $j=1,2,3,4$. (The entries in the matrices are
	exact up to four decimal places.) For the eigenvector $v$ in \eqref{eq:ev} we compute
	\[
	\gamma_7 v = - v.
	\]
	Thus, a symmetry breaking pitchfork bifurcation at $\mu = 0$ is induced by the group $\mathbf{Z}_2 = \{I, \gamma_7\}$
	rather than $\{ I , S \}$ (see \eqref{eq:S}), and this explains the phenomenon $(i)^*$
	discussed in Section~\ref{sec:Motivation}.
	
	Observe that $S$ is not among the matrices $\gamma_j$ $(j=1,\ldots,8)$, i.e. $S \not\in \Gamma_2(A)$,
	so that there is still
	another structure to be revealed. We will see in the following section that
	this is related to the fact that $4$ is a double eigenvalue of $A(0)$.
\end{example}

\section{Self-adjoint Matrices are Equivariant -- the General Case}
\label{sec:mainres}

In this section, we generalize Proposition~\ref{prop:main} significantly. In fact, we
will show that in general the group $\Sigma$ may be much more complex -- even
in the case where $A$ is not equivariant in the classical sense where
e.g.\ underlying geometric symmetries lead to
equvariance properties of a dynamical system.

\subsection{Orthogonal Isotropy Subgroups for Matrices}

The following observation forms the theoretical basis for our analytical investigations.
With this result we state a useful characterization of the group $\Gamma(A)$ containing
\emph{all} $\gamma \in \mathbf{O}(n)$ which commute with a given matrix $A$.

\begin{proposition} \label{thm:main}
	Let $A \in \mathbb{R}^{n,n}$ and $V \in \mathbf{O}(n)$. Define
	the compact group
	\begin{equation}
		\label{eq:Sigma}
		\Sigma_V(A) = \{ \sigma \in \mathbf{O}(n) \mid \sigma V A V^T = V A V^T \sigma \}
	\end{equation}
	and let
	\begin{equation} \label{eq:GammaV}
		\Gamma(A) = V^T \Sigma_V(A) V.
	\end{equation}
	Then for every $\gamma \in \mathbf{O}(n)$
	\begin{equation*}
		\gamma A = A \gamma\, \Longleftrightarrow \,\gamma \in \Gamma(A).
	\end{equation*}
	In particular, $\Gamma(A)$ does not depend on $V$, and we refer to $\Gamma(A)$ as the
	\emph{orthogonal isotropy subgroup of $A$}.
\end{proposition}

\begin{proof}
	Let  $\gamma \in \mathbf{O}(n)$ such that $\gamma A = A \gamma$. Then we have for each $V\in \mathbf{O}(n)$
	\begin{align*}
		\gamma A = A \gamma &\Longleftrightarrow V \gamma V^T V A V^T = V A V^T V \gamma V^T \\
		&\Longleftrightarrow \sigma V A V^T = V A V^T \sigma \text{ for } \sigma = V \gamma V^T \\
		&\Longleftrightarrow \gamma \in \Gamma(A). \qedhere
	\end{align*}
\end{proof}

\begin{remark} \label{rem:com} ~
	\begin{itemize}
		\item[(a)] Recall that the \emph{isotropy subgroup}
		for a point $x$ in some space $X$ characterizes the symmetry of $x$
		with respect to a certain group action.
		More precisely consider a group action $\vartheta$ of a group $G$ on
		a linear space $X$. Then the isotropy subgroup of $x\in X$ is given by
		$\{ g\in G : \vartheta(g) x = x\}$.
		
		If we let $G= \mathbf{O}(n)$ act on matrices $A \in X = \mathbb{R}^{n,n}$ by
		\[
		\vartheta(\gamma) A = \gamma A \gamma^T,
		\]
		then $\Gamma(A) \subset \mathbf{O}(n)$ in Proposition~\ref{thm:main} is the isotropy subgroup
		of $A$ with respect to this action. 
		\item[(b)] If we replace $\mathbf{O}(n)$ by $\mathbf{U}(n)$ (unitary matrices) or $\mathbf{GL}(n,\mathbb{R})$ (invertible matrices) and the
		matrix $V^T$ by $V^*$ or $V^{-1}$, respectively, then we obtain an analogous result for the unitary and invertible isotropy subgroup.
		It is also possible to formulate Proposition~\ref{thm:main} for general orthogonal,
		unitary or invertible operators.
	\end{itemize}
\end{remark}

We now show that $\Sigma_V(A)$ in \eqref{eq:Sigma} is unique up to orthogonal transformations.

\begin{corollary}
	Let $V \in \mathbf{O}(n)$.
	\begin{itemize}
		\item[(a)] For each $U \in \mathbf{O}(n)$ there exists $W \in \mathbf{O}(n)$ such that
		\[
		\Sigma_U(A) = W^T \Sigma_V(A) W.
		\]
		\item[(b)] For each $W \in \mathbf{O}(n)$ there exists $U \in \mathbf{O}(n)$ such that
		\[
		W^T \Sigma_V(A) W = \Sigma_U(A).
		\]
	\end{itemize}
	Thus,
	\begin{equation*}
		\{ \Sigma_U(A) : U \in \mathbf{O}(n) \} = \{ W^T \Sigma_V(A) W : W \in \mathbf{O}(n) \}.
	\end{equation*}
\end{corollary}

\begin{proof}
	Let $V \in \mathbf{O}(n)$ be given. Then, by Proposition~\ref{thm:main}, we have
	for $U \in \mathbf{O}(n)$
	\begin{equation*}
		V^T \Sigma_{V}(A) V = \Gamma(A) = U^T \Sigma_{U}(A) U
	\end{equation*}
	and therefore
	\begin{equation*}
		\Sigma_{U}(A) = (V U^T)^T \Sigma_{V}(A) (V U^T).
	\end{equation*}
	With $W = VU^T$ we obtain (a) and (b) follows by setting $U = W^T V$.
\end{proof}

\subsection{Equivariance Properties of Self-adjoint Matrices}
\label{subsec:EP}

We now return to the case where $A$ is self-adjoint. Our aim is to extend significantly
Proposition~\ref{prop:main}. This leads to the surprising fact that self-adjoint matrices may
possess hidden symmetries due to repeated eigenvalues. Denote by $\lambda_1 < \dots < \lambda_k$
the (real) sorted eigenvalues of $A$ with multiplicities $m = (m_1,...,m_k)$
and let $V \in \mathbf{O}(n)$ so that $V A V^T = D$, where $D\in \R^{n,n}$ is a diagonal matrix
containing the sorted eigenvalues $\lambda_i\in \R$ of $A$ on its diagonal.

\begin{definition}
	Let $k \leq n$ and $m \in \mathbb{N}^k$ so that $\sum_{i = 1}^k m_i = n$. Define $\mathbf{O_B}(m)$ to be the set of block-diagonal matrices where the $i$-th block is in $\mathbf{O}(m_i)$, i.e.
	\begin{equation*}
		\mathbf{O_B}(m) = \left\{ Q \in \mathbf{O}(n) : Q = \begin{pmatrix}
			Q_1 & 		 & \\
			& \ddots & \\
			&        & Q_k
		\end{pmatrix}
		\text{ with }
		Q_i \in \mathbf{O}(m_i) \quad \forall i = 1,...,k
		\right\}.
	\end{equation*}
\end{definition}
With this useful definition we are able to \emph{completely} characterize the symmetries of a self-adjoint matrix $A\in \R^{n,n}$.
\begin{corollary} \label{cor:sadj}
	Let $A \in \mathbb{R}^{n,n}$ be self-adjoint and $V \in \mathbf{O}(n)$ so that $V$ diagonalizes $A$ (and the eigenvalues on the diagonal are sorted). Then
	\begin{equation*}
		\Sigma_V(A) = \mathbf{O_B}(m),
	\end{equation*}
	where $m$ is the vector that contains the multiplicities of the eigenvalues of $A$.
	In particular, by Proposition \ref{thm:main} we have
	\begin{equation} \label{eq:GammaOB}
		\Gamma(A) = V^T \mathbf{O_B}(m) V.
	\end{equation}
\end{corollary}
\begin{proof}
	Using the fact that $V$ diagonalizes $A$ we can write $\Sigma_V(A)$ as
	\begin{align*}
		\Sigma_V(A) = \{ \sigma \in \mathbf{O}(n) \mid \sigma V A V^T = V A V^T \sigma \} =\{ \sigma \in \mathbf{O}(n) \mid \sigma D = D \sigma \}.
	\end{align*}
	Thus, we only need to show that $\sigma D= D\sigma$ is equivalent to the fact that $\sigma\in \Sigma_V$ is a block-diagonal matrix.
	Let $\sigma \in \Sigma_V(A)$ and write $\sigma = (\sigma_{i,j})_{i,j = 1,...,k}$ with rectangular blocks $\sigma_{i,j} \in \mathbb{R}^{m_i, m_j}$.
	Then we have $\sigma D = (\lambda_j \sigma_{i,j})_{i,j}$ and $D \sigma = (\lambda_i \sigma_{i,j})_{i,j}$. Therefore, $\sigma \in \Sigma_V$ translates into
	\begin{equation*}
		\sigma \in \mathbf{O}(n) \text{ and } \lambda_j \sigma_{i,j} = \lambda_i \sigma_{i,j} \quad \forall i,j = 1,...,k,
	\end{equation*}
	which is equivalent to $\sigma_{i,j} = 0$ for $i \neq j$ and $\sigma_{i,i} \in \mathbf{O}(m_i)$ for $i = 1,...,k$.
\end{proof}

\begin{remark} \label{rem:sadj} ~
	\begin{enumerate}
		\item[(a)] The order of the eigenvalues $\lambda_i\in\R$ is not relevant as long as $V$ is chosen in such a way
		that all instances of the same eigenvalue on the diagonal of $V A V^T$ are next to each other.
		Otherwise, the elements in $\Sigma_V$ are not block-diagonal.
		\item[(b)] Analogous results for normal matrices in the unitary case and diagonalizable matrices in the invertible case 
		(cf.\ Remark \ref{rem:com} (b)) follow in the same way.
		\item[(c)] Since $\mathbf{O}(1) = \{ +1,-1 \} = \mathbf{Z}_2$ and
		$\mathbf{Z}_2^\ell = \prod_{i = 1}^{\ell} \mathbf{Z}_2 \subseteq \mathbf{O_B}(\ell) = \mathbf{O}(\ell)$
		for every $\ell \in \mathbb{N}$, we have $\mathbf{Z}_2^n \subseteq \Sigma_V$ for each self-adjoint matrix $A$ (independently of the multiplicities of the eigenvalues of $A$). That is, we always have $\Gamma_2(A) \subset \Gamma(A)$
		(cf.\ Proposition \ref{prop:main} and Remark~\ref{rmk:main} (b)), and
		equality holds if and only if $A$ has only simple eigenvalues.
		In particular, $\Gamma(A)$ is finite if and only if $A$ has only simple eigenvalues.
	\end{enumerate}
\end{remark}

\begin{example} \label{ex:equi} ~
	\begin{itemize}
		\item[(a)]
		Let us return to our guiding example from Section~\ref{sec:Motivation} and consider the matrix
		\[
		A(-0.25) = \begin{pmatrix}
		2 & -\frac{3}{\sqrt{2}} & -\frac{3}{\sqrt{2}}\\[0.5em]
		-\frac{3}{\sqrt{2}} & \frac{7}{2} & -\frac{3}{2}\\[0.5em]
		-\frac{3}{\sqrt{2}} & -\frac{3}{2} & \frac{7}{2}
		\end{pmatrix}.
		\]
		The eigenvalues of $A(-0.25)$ are a simple eigenvalue $\lambda_1 = -1$ and a double eigenvalue
		$\lambda_2 = 5$.
		Thus, Corollary~\ref{cor:sadj} yields $\Sigma_V=\mathbf{O_B}(1,2)$
		and it turns out that $A(-0.25)$ is in fact $\mathbf{Z}_2\times \mathbf{O}(2)$-equivariant. Here
		$\mathbf{Z}_2 = \{ I, \gamma_7\}$ (cf.\ Example~\ref{ex:pitchfork}), and
		a reflection $S$ and a rotation $R$ by $\pi/2$ (exact up to four decimal places) within $\mathbf{O}(2)$ are given by
		\[
		S = \begin{pmatrix}
		1 & 0 & 0 \\
		0 & 0 & 1 \\
		0 & 1 & 0
		\end{pmatrix}
		\quad \mbox{and} \quad
		R = \begin{pmatrix}
		0.5000 & 0.8536 & -0.1464\\
		-0.1464 &  0.2500 & 0.9571\\
		0.8536 &-0.4571 & 0.2500
		\end{pmatrix}.
		\]
		It follows that if we have an equilibrium which is not $\mathbf{O}(2)$-symmetric, then we obtain
		an entire nontrivial $\mathbf{O}(2)$-orbit of equilibria. This explains the phenomenon described
		in $(ii)^*$ for the guiding example in Section~\ref{sec:Motivation}, see also Figure \ref{fig:BF} (b).
		
		Finally, observe that the bifurcating equilibria in Figure~\ref{fig:BF} (a) are $\mathbf{O}(2)$-symmetric
		and therefore isolated for each fixed value of $\mu$. Note that generically, it is not expected for a one-parameter family of self-adjoint matrices to possess multiple eigenvalues (see Appendix 10 in \cite{Arnold78}; or also \cite{DeMel94} where results of \cite{Arnold_1971} on general matrices have been extended to self-adjoint matrices in the equivariant context). Rather we have constructed this family for the purpose of illustration.
		\item[(b)]
		Consider the parameter-dependent family of matrices (see \cite{DeMel94})
		\begin{equation*}
			A(\mu)=\begin{pmatrix}
				D(\mu) & B(\mu) & 0 & B(\mu)\\
				B(\mu) & D(\mu) &B(\mu) & 0\\
				0 & B(\mu) &D(\mu) & B(\mu)\\
				B(\mu) & 0 &B(\mu) & D(\mu)
			\end{pmatrix} \in \R^{16,16}
		\end{equation*}
		with
		\begin{equation*}
			D(\mu)=\begin{pmatrix}
				-2.0 +\sin(\mu) & 0.2+\mu^2 & 0.4\mu &0.9\mu^2\\
				0.2+\mu^2 & -0.4 & -0.8+\mu(1-\mu) & \mu\sin(\mu)\\
				0.4\mu & -0.8+\mu(1-\mu) & -1.4 + \cos(\mu) & 0\\
				0.9\mu^2 & \mu\sin(\mu) & 0 & \mu
			\end{pmatrix}
		\end{equation*}
		and
		\begin{equation*}
			B(\mu)=\begin{pmatrix}
				1+\mu\cos(\mu) & -3.5\cos(\mu) &-0.5\mu & -1\\
				-3.5\cos(\mu) & -1+\mu & 0.5\mu^2 & 2+0.5\cos(\mu)\\
				-0.5\mu & 0.5 \mu^2 & 1+\mu & 0\\
				-1 & 2+0.5\cos(\mu) & 0 & \sin(\mu)
			\end{pmatrix}
		\end{equation*}
		for $\mu \in [-3,3]$. Then it is easy to verify that $A(\mu)$ is $\mathbf{D}_4$-equivariant. Here the action of the
		dihedral group $\mathbf{D}_4$ is generated by a rotation $R$ and a reflection $S$, where $R,S \in \R^{16,16}$.
		Written in a $4\times 4$-block structure these are given by
		\begin{equation*}
			R=\begin{pmatrix}
				0 & I_4 & 0 & 0\\
				0 & 0 & I_4 & 0\\
				0 & 0 & 0 & I_4\\
				I_4 & 0 & 0 & 0
			\end{pmatrix}
			\quad \mbox{and} \quad
			S=\begin{pmatrix}
				I_4 & 0 & 0 & 0\\
				0 & 0 & 0 & I_4\\
				0 & 0 & I_4 & 0\\
				0 & I_4 & 0 & 0
			\end{pmatrix}.
		\end{equation*}
		However, $A(\mu)$ is self-adjoint and therefore this matrix family is not just
		$\mathbf{D}_4$- but even	$\Gamma(A)$-equivariant where $\Gamma(A)$ is given in
		\eqref{eq:GammaOB}. Now $A(\mu)$ has $8$ simple and $4$ double eigenvalues,
		where the multiple eigenvalues are induced by the two-dimensional irreducible
		representations of $\mathbf{D}_4$ (see \cite{GSS88}). Hence, 
		for each parameter value $A(\mu)$ is indeed $\mathbf{Z}_2^{8}\otimes \mathbf{O}(2)^4$-equivariant
		by Corollary \ref{cor:sadj} and~\eqref{eq:GammaOB}.
		
		Finally, we explicitly list for $\mu = 0$ a couple of elements of $\Gamma(A)$
		and $\Sigma_V(A)$ for illustration purposes. For instance, the matrices $R$ and $S$
		above are given by
		\begin{equation*}
			R = V^\top \sigma_R V \quad \text{and} \quad S = V^\top \sigma_S V,
		\end{equation*}
		where $\sigma_R, \sigma_S  \in \mathbf{O_B}(1,1,1,2,1,2,2,2,1,1,1,1)$ and
		{ \scriptsize
			\begin{equation*}
				\sigma_R = \begin{pmatrix}
					-1 &   &    &    &   &   &   &    &    &   &   &    &    &   &    &   \\
					& 1 &    &    &   &   &   &    &    &   &   &    &    &   &    &   \\
					&   & -1 &    &   &   &   &    &    &   &   &    &    &   &    &   \\
					&   &    & 0  & 1 &   &   &    &    &   &   &    &    &   &    &   \\
					&   &    & -1 & 0 &   &   &    &    &   &   &    &    &   &    &   \\
					&   &    &    &   & 1 &   &    &    &   &   &    &    &   &    &   \\
					&   &    &    &   &   & 0 & -1 &    &   &   &    &    &   &    &   \\
					&   &    &    &   &   & 1 & 0  &    &   &   &    &    &   &    &   \\
					&   &    &    &   &   &   &    & 0  & 1 &   &    &    &   &    &   \\
					&   &    &    &   &   &   &    & -1 & 0 &   &    &    &   &    &   \\
					&   &    &    &   &   &   &    &    &   & 0 & -1 &    &   &    &   \\
					&   &    &    &   &   &   &    &    &   & 1 & 0  &    &   &    &   \\
					&   &    &    &   &   &   &    &    &   &   &    & -1 &   &    &   \\
					&   &    &    &   &   &   &    &    &   &   &    &    & 1 &    &   \\
					&   &    &    &   &   &   &    &    &   &   &    &    &   & -1 &   \\
					&   &    &    &   &   &   &    &    &   &   &    &    &   &    & 1
				\end{pmatrix}
		\end{equation*} }
		and
		{ \scriptsize
			\begin{equation*}
				\setlength\arraycolsep{4pt}
				\sigma_S = \begin{pmatrix}
					1 &   &    &         &         &   &         &        &         &         &         &        &   &   &   &   \\
					& 1 &    &         &         &   &         &        &         &         &         &        &   &   &   &   \\
					&   &  1 &         &         &   &         &        &         &         &         &        &   &   &   &   \\
					&   &    & -0.6065 & -0.7951 &   &         &        &         &         &         &        &   &   &   &   \\
					&   &    & -0.7951 & 0.6065  &   &         &        &         &         &         &        &   &   &   &   \\
					&   &    &         &         & 1 &         &        &         &         &         &        &   &   &   &   \\
					&   &    &         &         &   & -0.8442 & 0.5361 &         &         &         &        &   &   &   &   \\
					&   &    &         &         &   & 0.5361  & 0.8442 &         &         &         &        &   &   &   &   \\
					&   &    &         &         &   &         &        & 0.9699  & -0.2434 &         &        &   &   &   &   \\
					&   &    &         &         &   &         &        & -0.2434 & -0.9699 &         &        &   &   &   &   \\
					&   &    &         &         &   &         &        &         &         & -0.8607 & 0.5091 &   &   &   &   \\
					&   &    &         &         &   &         &        &         &         & 0.5091  & 0.8607 &   &   &   &   \\
					&   &    &         &         &   &         &        &         &         &         &        & 1 &   &   &   \\
					&   &    &         &         &   &         &        &         &         &         &        &   & 1 &   &   \\
					&   &    &         &         &   &         &        &         &         &         &        &   &   & 1 &   \\
					&   &    &         &         &   &         &        &         &         &         &        &   &   &   & 1
				\end{pmatrix}.
		\end{equation*} }
		A couple of `hidden symmetries' -- i.e.\ elements of $\Gamma(A)$ which are not contained
		in $\mathbf{D}_4$~-- are given by
		{ \scriptsize
			\begin{equation*}
				\gamma_1 = \begin{pmatrix}
					1 &   &   &   &   &   &   &   &   &   &   &   &   &   &   &   \\
					& 1 &   &   &   &   &   &   &   &   &   &   &   &   &   &   \\
					&   & 1 &   &   &   &   &   &   &   &   &   &   &   &   &   \\
					&   &   & 0 &   &   &   &   &   &   &   & 1 &   &   &   &   \\
					&   &   &   & 1 &   &   &   &   &   &   &   &   &   &   &   \\
					&   &   &   &   & 1 &   &   &   &   &   &   &   &   &   &   \\
					&   &   &   &   &   & 1 &   &   &   &   &   &   &   &   &   \\
					&   &   &   &   &   &   & 0 &   &   &   &   &   &   &   & 1 \\
					&   &   &   &   &   &   &   & 1 &   &   &   &   &   &   &   \\
					&   &   &   &   &   &   &   &   & 1 &   &   &   &   &   &   \\
					&   &   &   &   &   &   &   &   &   & 1 &   &   &   &   &   \\
					&   &   & 1 &   &   &   &   &   &   &   & 0 &   &   &   &   \\
					&   &   &   &   &   &   &   &   &   &   &   & 1 &   &   &   \\
					&   &   &   &   &   &   &   &   &   &   &   &   & 1 &   &   \\
					&   &   &   &   &   &   &   &   &   &   &   &   &   & 1 &   \\
					&   &   &   &   &   &   & 1 &   &   &   &   &   &   &   & 0 
				\end{pmatrix},
			\end{equation*}
			\begin{equation*}
				\sigma_1 = \begin{pmatrix}
					1 &   &   &   &   &   &   &   &    &    &   &   &   &   &   &   \\
					& 1 &   &   &   &   &   &   &    &    &   &   &   &   &   &   \\
					&   & 1 &   &   &   &   &   &    &    &   &   &   &   &   &   \\
					&   &   & 1 & 0 &   &   &   &    &    &   &   &   &   &   &   \\
					&   &   & 0 & 1 &   &   &   &    &    &   &   &   &   &   &   \\
					&   &   &   &   & 1 &   &   &    &    &   &   &   &   &   &   \\
					&   &   &   &   &   & 1 & 0 &    &    &   &   &   &   &   &   \\
					&   &   &   &   &   & 0 & 1 &    &    &   &   &   &   &   &   \\
					&   &   &   &   &   &   &   & -1 & 0  &   &   &   &   &   &   \\
					&   &   &   &   &   &   &   & 0  & -1 &   &   &   &   &   &   \\
					&   &   &   &   &   &   &   &    &    & 1 & 0 &   &   &   &   \\
					&   &   &   &   &   &   &   &    &    & 0 & 1 &   &   &   &   \\
					&   &   &   &   &   &   &   &    &    &   &   & 1 &   &   &   \\
					&   &   &   &   &   &   &   &    &    &   &   &   & 1 &   &   \\
					&   &   &   &   &   &   &   &    &    &   &   &   &   & 1 &   \\
					&   &   &   &   &   &   &   &    &    &   &   &   &   &   & 1 
				\end{pmatrix}.
		\end{equation*} }
		or
		{ \scriptsize
			\begin{equation*} 
				\gamma_2 = \begin{pmatrix}
					1 &   &   &         &   &   &   &         &   &   &   &         &   &   &   &         \\
					& 1 &   &         &   &   &   &         &   &   &   &         &   &   &   &         \\
					&   & 1 &         &   &   &   &         &   &   &   &         &   &   &   &         \\
					&   &   & 0.3783  &   &   &   & 0.4850  &   &   &   & 0.6217  &   &   &   & -0.4850 \\
					&   &   &         & 1 &   &   &         &   &   &   &         &   &   &   &         \\
					&   &   &         &   & 1 &   &         &   &   &   &         &   &   &   &         \\
					&   &   &         &   &   & 1 &         &   &   &   &         &   &   &   &         \\
					&   &   & 0.4850  &   &   &   & 0.6217  &   &   &   & -0.4850 &   &   &   & 0.3783  \\
					&   &   &         &   &   &   &         & 1 &   &   &         &   &   &   &         \\
					&   &   &         &   &   &   &         &   & 1 &   &         &   &   &   &         \\
					&   &   &         &   &   &   &         &   &   & 1 &         &   &   &   &         \\
					&   &   & 0.6217  &   &   &   & -0.4850 &   &   &   & 0.3783  &   &   &   & 0.4850  \\
					&   &   &         &   &   &   &         &   &   &   &         & 1 &   &   &   	   \\
					&   &   &         &   &   &   &         &   &   &   &         &   & 1 &   &         \\
					&   &   &         &   &   &   &         &   &   &   &         &   &   & 1 &         \\
					&   &   & -0.4850 &   &   &   & 0.3783  &   &   &   & 0.4992  &   &   &   & 0.6217 
				\end{pmatrix},
			\end{equation*}
			\begin{equation*}
				\sigma_2 = \begin{pmatrix}
					1 &   &   &   &   &   &   &   &   &   &   &   &   &   &   &   \\
					& 1 &   &   &   &   &   &   &   &   &   &   &   &   &   &   \\
					&   & 1 &   &   &   &   &   &   &   &   &   &   &   &   &   \\
					&   &   & 1 & 0 &   &   &   &   &   &   &   &   &   &   &   \\
					&   &   & 0 & 1 &   &   &   &   &   &   &   &   &   &   &   \\
					&   &   &   &   & 1 &   &   &   &   &   &   &   &   &   &   \\
					&   &   &   &   &   & 1 & 0 &   &   &   &   &   &   &   &   \\
					&   &   &   &   &   & 0 & 1 &   &   &   &   &   &   &   &   \\
					&   &   &   &   &   &   &   & 0 & 1 &   &   &   &   &   &   \\
					&   &   &   &   &   &   &   & 1 & 0 &   &   &   &   &   &   \\
					&   &   &   &   &   &   &   &   &   & 1 & 0 &   &   &   &   \\
					&   &   &   &   &   &   &   &   &   & 0 & 1 &   &   &   &   \\
					&   &   &   &   &   &   &   &   &   &   &   & 1 &   &   &   \\
					&   &   &   &   &   &   &   &   &   &   &   &   & 1 &   &   \\
					&   &   &   &   &   &   &   &   &   &   &   &   &   & 1 &   \\
					&   &   &   &   &   &   &   &   &   &   &   &   &   &   & 1 
				\end{pmatrix}.
		\end{equation*} }
	\end{itemize}
\end{example}

\section{Implications for Equivariant Dynamical Systems}
\label{sec:Impl}

In this section, we discuss by an example the implications for dynamical systems of the form
\begin{equation} \label{eq:DS}
	\dot{x} = A(\mu) x-f(x,\mu),
\end{equation}
where $A(\mu)\in \R^{n,n}$ is self-adjoint for all $\mu\in\R$ and
$f \colon \R^n \times \R \to \R^n$ is $\mathbf{O}(n)$-equivariant, i.e.
\begin{equation}
	\label{eq:fGamma}
	f(\gamma x,\mu)=\gamma f(x,\mu) \quad\forall \gamma \in \mathbf{O}(n).
\end{equation}
It follows that the right-hand side in \eqref{eq:DS} is $\Gamma(A(\mu))$-equivariant. In particular, the symmetry group varies with the parameter $\mu$, and a detailed bifurcation analysis in this context should be developed elsewhere. Here, we rather focus on the description of qualitative dynamical phenomena induced by the hidden symmetries.

\begin{remark} ~
	\begin{itemize}
		\item[(a)]Observe that the requirement on $f$ is satisfied if, for instance,
		\begin{equation}
			\label{eq:fgx}
			f(x,\mu) = g(x,\mu) x
		\end{equation}
		where $g\colon\R^n \times \R \to \R$ is $\mathbf{O}(n)$-invariant, that is
		$g(\gamma x,\mu)=g(x,\mu)$ for all $ \gamma \in \mathbf{O}(n)$.
		In particular, the equivariance condition \eqref{eq:fGamma} would hold for $g(x,\mu) = h(x)$ or
		$g(x,\mu) = h(A(\mu) x)$ where $h$ is $\mathbf{O}(n)$-invariant.
		\item[(b)] If we consider e.g.\ the \emph{nonlinear Schr\"odinger/Gross--Pitaevskii equation} \cite{DMM15} or the \emph{cubic Ginzburg--Landau equation} \cite{AK02} in two dimensions
		then a numerical discretization by the method of lines yields a dynamical system of the form \eqref{eq:DS}, where $A$ does not explicitly depend on $\mu$.
		Moreover, if the underlying spatial domain is $\mathbf{D}_\ell$-symmetric ($\ell \ge 3$) then 
		symmetry related bifurcations of the system will be governed by
		$\Gamma(A)$ rather than just $\mathbf{D}_\ell$. In fact, in this case we expect to observe phenomena 
		driven by the hidden symmetries as already described in Example~\ref{ex:equi}~(b).
	\end{itemize}
	\end {remark}
	
	As a concrete example, we consider our guiding example introduced in Section~\ref{sec:Motivation} and let
	\begin{equation*}
		A(\mu) = \begin{pmatrix}
			2 & \sqrt{2}(2 \mu - 1) & \sqrt{2}(2 \mu - 1) \\
			\sqrt{2}(2 \mu - 1) & 3 - 2 \mu & 2 \mu - 1 \\
			\sqrt{2}(2 \mu - 1) & 2 \mu - 1 & 3 - 2 \mu
		\end{pmatrix}
	\end{equation*}
	and $g(x,\mu) = \norm{x}_2^2$ (see \eqref{eq:fgx}).
	The eigenvalues of $A(\mu)$ are $\lambda_1(\mu) = 4 \mu$ with multiplicity $1$ and $\lambda_2(\mu) = 4(1 - \mu)$ with multiplicity $2$.
	
	In the following analytic considerations, we will use the fact that
	a point $x^* \not= 0$ is an equilibrium of \eqref{eq:DS} if and only if $g(x^*,\mu^*)$ is an eigenvalue of $A(\mu)$
	and $x^*$ is a corresponding eigenvector.
	This follows immediately from the structure of \eqref{eq:DS}.
	For $g(x)=\norm{x}_2^2$ this means that every appropriately scaled eigenvector of a positive eigenvalue of $A(\mu)$ is an equilibrium and vice versa (i.e.\ every equilibrium is an appropriately scaled eigenvector of $A(\mu)$).

	Since $0$ is always an equilibrium it will be omitted in the following considerations.
	By the results in this work, we immediately know how the set of equilibria of \eqref{eq:DS} changes with respect to $\mu$.
	Observe that if $x^*$ is an equilibrium, then the $\Gamma(A(\mu))$-equivariance implies that $\gamma x^*\in\R^n$ is
	an equilibrium for all $\gamma \in \Gamma(A(\mu))$.
	
	\begin{itemize}
		\item $\mu < 0$: There is a circle of equilibria induced by the $\mathbf{O}(2)$ equivariance of
		$f$ (see Example~\ref{ex:equi}~(a)).
		\item $\mu = 0$: There still is a circle of equilibria and a pitchfork bifurcation occurs (cf.\ Section~\ref{sec:Motivation}).
		\item $\mu \in (0,0.5)$: There is a circle of equilibria and two isolated equilibrium points.
		\item $\mu = 0.5$: Since $A(0.5)$ possesses the threefold eigenvalue $\lambda_1(0.5) = \lambda_2(0.5) = 2$
		the set of equilibria becomes a sphere.
		\item $\mu\in(0.5,1)$: The sphere breaks up and there is again a circle of equilibria and two isolated equilibrium points.
		\item $\mu = 1$: A subcritical pitchfork bifurcation occurs, and the circle of equilibria disappears.
		\item $\mu > 1$: There are only two equilibria left.
	\end{itemize}
	Figure \ref{fig:ex} shows the set of equilibria for different values of $\mu$.
	
	\begin{figure}[htb]
		\centering
		\begin{minipage}[t]{0.49\textwidth}
			\centering
			\subfiguretitle{(a) $ \mu = -0.25 $}
			\includegraphics[width = 0.95\textwidth]{pics/mue_-025}
		\end{minipage}
		\begin{minipage}[t]{0.49\textwidth}
			\centering
			\subfiguretitle{(b) $ \mu = 0.01 $}
			\includegraphics[width = 0.95\textwidth]{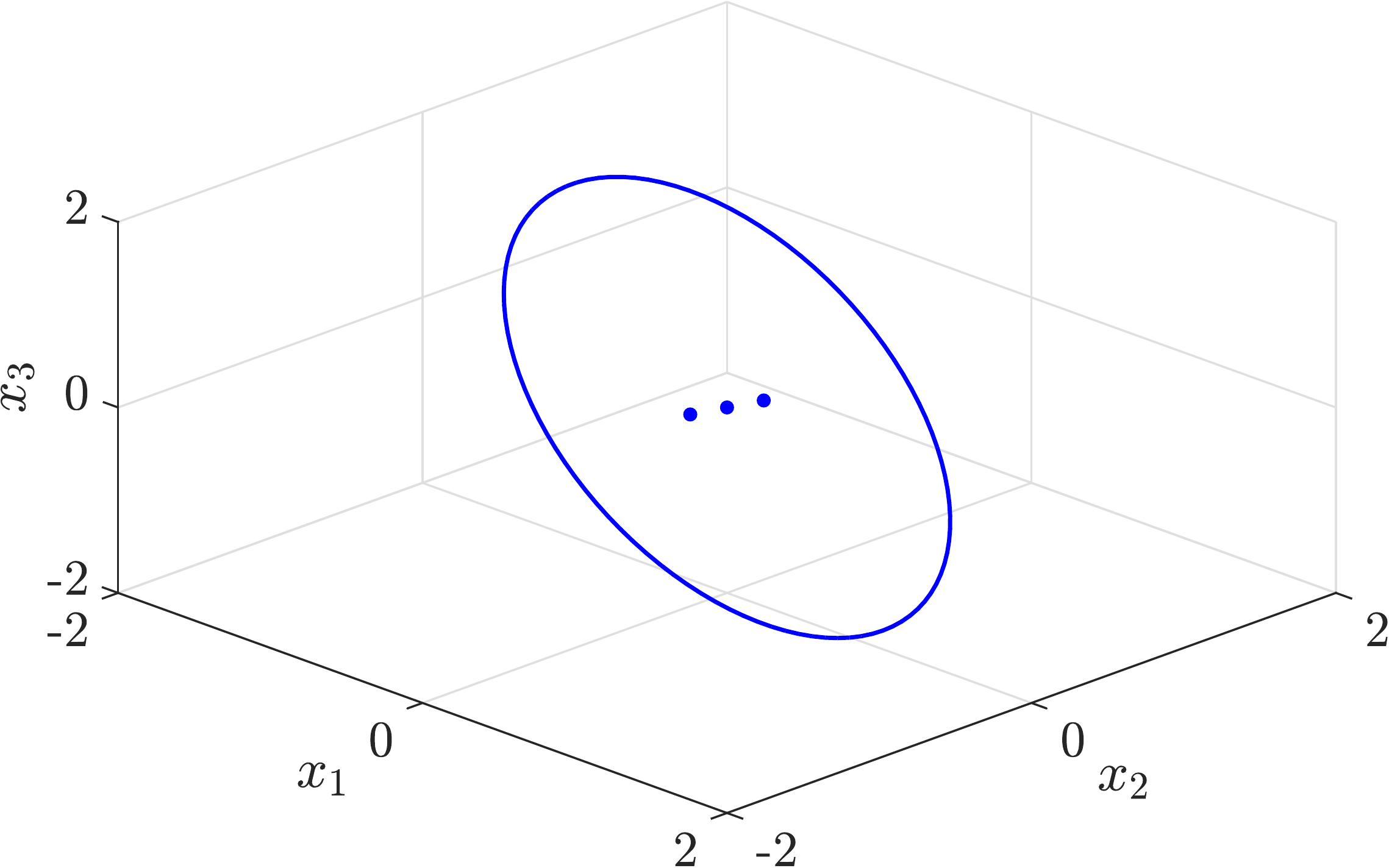}
		\end{minipage} \\[2ex]
		\begin{minipage}[t]{0.49\textwidth}
			\centering
			\subfiguretitle{(c) $ \mu = 0.99 $}
			\includegraphics[width = 0.95\textwidth]{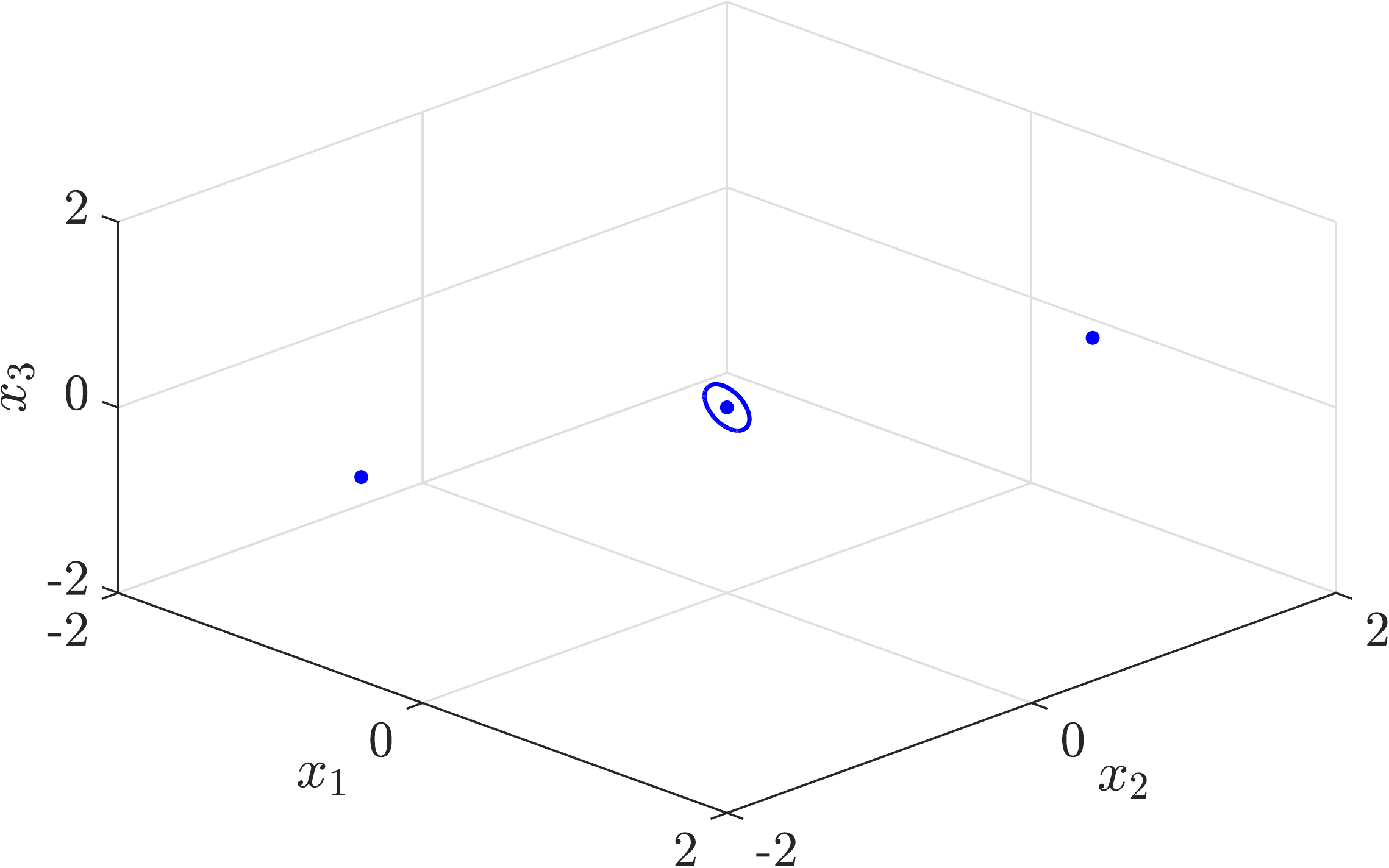}
		\end{minipage}
		\begin{minipage}[t]{0.49\textwidth}
			\centering
			\subfiguretitle{(d) $ \mu = 1.25 $}
			\includegraphics[width = 0.95\textwidth]{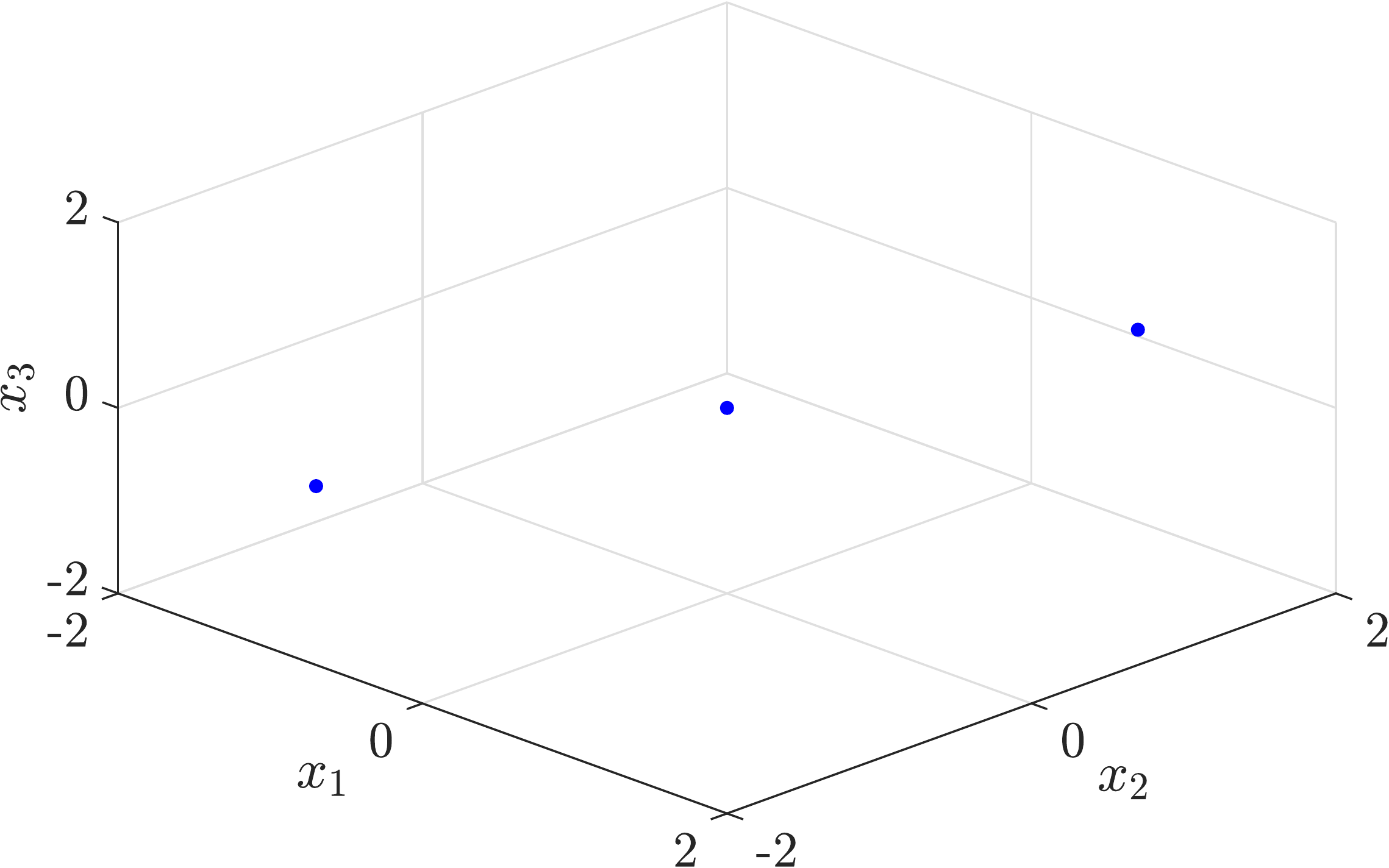}
		\end{minipage}
		\caption{The set of equilibria of the dynamical system \eqref{eq:DS} for different values of $\mu$.}
		\label{fig:ex}
	\end{figure}
	
	\section{Other Applications}
	\label{sec:appl}
	
	In addition to the implications for symmetry breaking bifurcation phenomena
	as illustrated in Section \ref{sec:Impl}, our results have further applications. We
	illustrate this briefly by the following three examples.
	
	\subsection{Two-sided orthogonal Procrustes problem}
	\label{sec:app1}
	
	Given two symmetric matrices $A, B \in \R^{n, n}$, the two-sided orthogonal Procrustes problem can be defined as follows: Find an orthogonal matrix $ P \in \mathbf{O}(n) $ such that the cost function $ \norm{P A - B P}_F $ is minimized. It is well known -- see, for instance, \cite{Scho68,HRW92,AW00} -- that an optimal solution is given by $ P = V_B^T V_A $, where $ D_A = V_A A V_A^T $ and $ D_B = V_B B V_B^T $ are the eigendecompositions of $ A $ and $ B $, respectively. Note that the eigenvalues in $ D_A $ and $ D_B $ both have to be sorted in nonincreasing (or, alternatively, nondecreasing) order. If the cost function is to be maximized, the eigenvalues need to be ordered in opposite order. With the aid of the results from Section~\ref{sec:mainres}, we can now characterize all solutions of this form, i.e.
	\begin{equation*}
		P \in \Sigma_{A,B} = \Bigl\{ V_B^T \sigma_B^T \sigma_A V_A \, \Bigm \vert \, \sigma_A \in \Sigma_{V_A}(A) \text{ and } \sigma_B \in \Sigma_{V_B}(B) \Bigr\},
	\end{equation*}
	since for such $ P $ we obtain
	\begin{align*}
		\norm{P A - B P}_F &= \norm{V_B^T \sigma_B^T \sigma_A V_A V_A^T D_A V_A - V_B^T D_B V_B V_B^T \sigma_B^T \sigma_A V_A}_F \\
		&= \norm{\sigma_B^T \sigma_A D_A - D_B \sigma_B^T \sigma_A}_F \\
		&= \norm{\sigma_B^T D_A \sigma_A - \sigma_B^T D_B \sigma_A}_F \\
		&= \norm{D_A - D_B}_F,
	\end{align*}
	which is indeed the optimal solution. Here, we used the invariance of the Frobenius norm under unitary transformations and the equivariance properties.
	
	\begin{remark} \label{rem:opp} ~
		\begin{enumerate}[label=(\alph*)]
			\item If $ D_A = D_B $, then $ \sigma_B^T \sigma_A \in \Sigma_{V_A}(A) $ and it suffices to consider matrices of the form $ P = V_B^T \sigma V_A $ for $ \sigma \in \Sigma_{V_A}(A) $ (or, equivalently, $ \sigma \in \Sigma_{V_B}(B) $).
			\item If, furthermore, all eigenvalues are distinct, then the eigenvectors are determined up to the sign and we obtain the group $ \Sigma $ and thus the special case derived in \cite{Scho68}.
			\item Since minimizing the Procrustes cost function corresponds to maximizing the cost function $ \tr(A^T P^T B P) $ and vice versa, the results can be extended to the orthogonal relaxation of the \emph{quadratic assignment problem (QAP)} \cite{HRW92,AW00}.
			\item Given two undirected graphs $G_A$ and $G_B$ with adjacency matrices $ A $ and $ B $, the graphs are isomorphic if they are isospectral and $ \Sigma_{A,B} $ contains a permutation matrix, see also \cite{KS18}.
		\end{enumerate}
	\end{remark}
	
	\subsection{Graph Symmetries} \label{sec:app2}
	An isomorphism from a graph to itself is called an automorphism. Let $A$ be the adjacency matrix of an undirected graph $G_A$, then the \textit{automorphism group} (or \textit{symmetry group}) of $G_A$ is defined as
	\begin{equation*}
		\Aut(G_A) = \big\{ PA = AP \, \big\vert \, P \text{ permutation matrix} \big\}.
	\end{equation*}
	A graph $G_A$ is called \textit{asymmetric} if $\Aut(G_A)$ is trivial, i.e.\ $\Aut(G_A) = \{ \mathrm{Id} \}$. Since $A$ is self-adjoint, we can use Corollary~\ref{cor:sadj} to identify the orthogonal commutator $\Gamma(A)$ of~$A$. Permutation matrices are orthogonal, hence $\Aut(G_A) \subseteq \Gamma(A)$.
	
	Our results show that even in the case where the graph $G_A$ is asymmetric it typically possesses additional symmetries -- namely the elements of the group $\Gamma(A)$.
	We illustrate this with the following example (cf.\ \cite{FS2015}, Figure 5):
	\begin{example}
		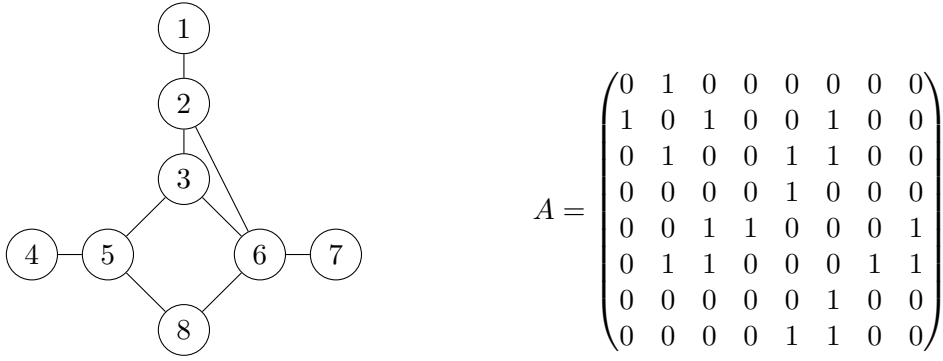
\begin{figure}
			\centering
			\begin{minipage}[b]{0.4\textwidth}
				\centering
				\begin{tikzpicture}
				\node[shape=circle,draw=black] (A) at (0,0) {$1$};
				\node[shape=circle,draw=black] (B) at (0,-1) {$2$};
				\node[shape=circle,draw=black] (C) at (0,-2) {$3$};
				\node[shape=circle,draw=black] (D) at (-2,-3) {$4$};
				\node[shape=circle,draw=black] (E) at (-1,-3) {$5$};
				\node[shape=circle,draw=black] (F) at (1,-3) {$6$} ;
				\node[shape=circle,draw=black] (G) at (2,-3) {$7$};
				\node[shape=circle,draw=black] (H) at (0,-4) {$8$};
				
				\path [-](A) edge node {} (B);
				\path [-](B) edge node {} (C);
				\path [-](B) edge node {} (F);
				\path [-](C) edge node {} (E);
				\path [-](C) edge node {} (F);
				\path [-](D) edge node {} (E);
				\path [-](E) edge node {} (H);
				\path [-](F) edge node {} (G);
				\path [-](F) edge node {} (H);
				\end{tikzpicture}
			\end{minipage}
			\begin{minipage}[b]{0.5\textwidth}
				\centering
				\begin{align*}
					A &= \begin{pmatrix}
						0 & 1 & 0 & 0 & 0 & 0 & 0 & 0 \\
						1 & 0 & 1 & 0 & 0 & 1 & 0 & 0 \\
						0 & 1 & 0 & 0 & 1 & 1 & 0 & 0 \\
						0 & 0 & 0 & 0 & 1 & 0 & 0 & 0 \\
						0 & 0 & 1 & 1 & 0 & 0 & 0 & 1 \\
						0 & 1 & 1 & 0 & 0 & 0 & 1 & 1 \\
						0 & 0 & 0 & 0 & 0 & 1 & 0 & 0 \\
						0 & 0 & 0 & 0 & 1 & 1 & 0 & 0
					\end{pmatrix}
				\end{align*}
			\end{minipage}
			\caption{Asymmetric graph and its adjacency matrix.}
			\label{fig:Graph}
		\end{figure}
		Consider the graph $G_A$ with adjacency matrix $A$ shown in Figure~\ref{fig:Graph}. The graph is asymmetric (due to the edge between the vertices $2$ and $6$). The eigenvalues are
		$(\lambda_1,\ldots,\lambda_7) = ( -2.24, -1.66, -0.83, 0, 0.74, 1.29, 2.70 )$ with
		corresponding multiplicities $m = (1,1,1,2,1,1,1)$. (Here we use the notation from Section~\ref{subsec:EP}.)
		By Corollary~\ref{cor:sadj}, we know that if $V$ is the orthogonal matrix that diagonalizes $A$,
		then all $\gamma \in \Gamma(A)$ are of the form 
		\begin{equation*}
			\gamma = V^T \sigma V = V^T \begin{pmatrix}
				\pm 1 &  &  &  &  &  &  &  \\
				& \pm 1 & & & & & & \\
				& & \pm 1 & & & & & \\
				& & & Q & & & & \\
				& & & & \pm 1 & & & \\
				& & & & & \pm 1 & & \\
				& & & & & & \pm 1 &
			\end{pmatrix} V, \quad \sigma \in \Sigma_V(A),
		\end{equation*}
		with an arbitrary $Q \in \mathbf{O}(2)$. For instance, we have $\gamma A = A \gamma$ for
		{ \footnotesize
			\begin{align*}
				\gamma = \begin{pmatrix}
					0.25  & 0 & 0.75 & -0.25 & 0 & 0 & -0.25 & -0.5 \\
					0 	  & 1 & 0    & 0     & 0 & 0 & 0     & 0    \\
					0.75  & 0 & 0.25 & 0.25  & 0 & 0 & 0.25  & 0.5  \\
					-0.25 & 0 & 0.25 & 0.25  & 0 & 0 & -0.75 & 0.5  \\
					0     & 0 & 0    & 0     & 1 & 0 & 0     & 0    \\
					0     & 0 & 0    & 0     & 0 & 1 & 0     & 0    \\
					-0.25 & 0 & 0.25 & -0.75 & 0 & 0 & 0.25  & 0.5  \\
					-0.5  & 0 & 0.5  & 0.5   & 0 & 0 & 0.5   & 0    
				\end{pmatrix} \text{ with }  
				\sigma = \begin{pmatrix}
					1 &   &   &    &    &   &   &   \\
					& 1 &   &    &    &   &   &   \\
					&   & 1 &    &    &   &   &   \\
					&   &   & -1 &    &   &   &   \\
					&   &   &    & -1 &   &   &   \\
					&   &   &    &    & 1 &   &   \\
					&   &   &    &    &   & 1 &   \\
					&   &   &    &    &   &   & 1 \\
				\end{pmatrix},
		\end{align*} }
		or for
		{ \footnotesize
			\begin{align*}
				\gamma = \begin{pmatrix}
					0.625   & 0 & 0.375   & 0.2286  & 0 & 0 & 0.2286  & -0.6036 \\
					0 	    & 1 & 0       & 0       & 0 & 0 & 0       & 0       \\
					0.375   & 0 & 0.625   & -0.2286 & 0 & 0 & -0.2286 & 0.6036  \\
					-0.4786 & 0 & 0.4786  & 0.625   & 0 & 0 & -0.375  & -0.1036 \\
					0       & 0 & 0       & 0       & 1 & 0 & 0       & 0       \\
					0       & 0 & 0       & 0       & 0 & 1 & 0       & 0       \\
					-0.4786 & 0 & 0.4786  & -0.375  & 0 & 0 & 0.625   & -0.1036 \\
					0.1036  & 0 & -0.1036 & 0.6036  & 0 & 0 & 0.6036  & 0.5    
				\end{pmatrix} 
				\text{ with } 
				\sigma = \begin{pmatrix}
					1 &   &   &    &   &   &   &   \\
					& 1 &   &    &   &   &   &   \\
					&   & 1 &    &   &   &   &   \\
					&   &   & 0  & 1 &   &   &   \\
					&   &   & -1 & 0 &   &   &   \\
					&   &   &    &   & 1 &   &   \\
					&   &   &    &   &   & 1 &   \\
					&   &   &    &   &   &   & 1 \\
				\end{pmatrix}.
		\end{align*} }
	\end{example}
	
	\subsection{Taylor Expansions}
	\label{sec:app3}
	
	Finally, let us briefly mention one implication involving Taylor expansions.
	In fact, in this context our main result Corollary \ref{cor:sadj} can be used to develop a
	novel general technique for the construction of higher order stencils for
	real valued functions of several variables.
	
	Suppose that $f\colon\R^n \to \R$ is smooth in a neighborhood of
	$\bar x\in\R^n$. In what follows, we use Corollary \ref{cor:sadj} to
	construct a four-point stencil which provides a second order approximation of
	evaluations of the fourth-order derivative in $\bar x$. For convenience,
	we write the Taylor expansion of $f$ in $\bar x$ as
	\[
	f(\bar x+h) = f(\bar x) + \nabla f(\bar x)^T h + \frac{1}{2} h^T H(\bar x) h + \sum_{j=3}^\infty g_j(\bar x,h),
	\]
	where $g_j(\bar x,h) = O(\| h\|^j)$, $j=3,4,\ldots$, and $H(\bar x)$ is the Hessian matrix of $f$ at $\bar x$.
	\begin{corollary}
		Denote by $\Gamma(\bar x)$ the group in \eqref{eq:GammaOB}
		corresponding to the Hessian matrix $H(\bar x)$.
		Then for all $\gamma \in\Gamma(\bar x)$ we have
		\begin{equation}
			\label{eq:Taylor1}
			f(\bar x+\gamma h) -2f(\bar x)+f(\bar x-\gamma h) = h^T H(\bar x) h + 2 g_4(\bar x,\gamma h) + O(\| h\|^6),
		\end{equation}
		and therefore for all $\gamma_1, \gamma_2\in\Gamma(\bar x)$
		\begin{equation}\label{eq:Taylor2}
			\begin{split}
				& ~ f(\bar x+\gamma_1 h) + f(\bar x-\gamma_1 h) -  f(\bar x+\gamma_2 h) - f(\bar x-\gamma_2 h) \\
				=& ~ 2(g_4(\bar x,\gamma_1 h) - g_4(\bar x,\gamma_2 h)) + O(\| h\|^6).
			\end{split}
		\end{equation}
		In particular, $f(\bar x+\gamma_1 h) + f(\bar x-\gamma_1 h) -  f(\bar x+\gamma_2 h) - f(\bar x-\gamma_2 h) = O(\| h\|^4)$.
	\end{corollary}
	
	\begin{proof}
		For $h\in\R^n$ and $\gamma_j \in\Gamma(\bar x)$ $(j=1,2)$ we compute
		\[
		f(\bar x\pm \gamma_j h) = 
		f(\bar x) \pm \nabla f(\bar x)^T \gamma_j h + \frac{1}{2} h^T H(\bar x) h \pm g_3(\bar x,\gamma_j h) + g_4(\bar x,\gamma_j h) \pm g_5(\bar x,\gamma_j h) + \cdots
		\]
		Therefore, using the fact that $\Gamma(\bar x) \subset \mathbf{O}(n)$
		\begin{eqnarray*}
			f(\bar x+\gamma_1 h) +f(\bar x-\gamma_1 h) & = & 2 \left( f(\bar x)  + \frac{1}{2} h^T H(\bar x) h + g_4(\bar x,\gamma_1 h) + O(\| h\|^6)\right), \\
			f(\bar x+\gamma_2 h) +f(\bar x-\gamma_2 h) & = & 2 \left( f(\bar x)  + \frac{1}{2} h^T H(\bar x) h + g_4(\bar x,\gamma_2 h) + O(\| h\|^6)\right), 
		\end{eqnarray*}
		and \eqref{eq:Taylor1}, \eqref{eq:Taylor2} immediately follow.
	\end{proof}
	
	Obviously, if $\gamma_1 = \pm \gamma_2$ then
	this result is not useful. However, for all other
	choices of $\gamma_j$ this leads to interesting approximations of the fourth-order derivative
	as long as $h$ is not an eigenvector of $\gamma_j$ ($j=1,2$).
	
	\begin{example}
		Let $f\colon\R^3 \to \R$ be defined by
		\[
		f(x_1,x_2,x_3) = x_1 x_2 x_3^2 +x_1^2 - 3x_2^2 + x_2 \sin(x_1) - x_2^2 x_3^2.
		\]
		We choose $\bar x = (1,1,1)^T$ and compute
		\[
		H(\bar x) = \begin{pmatrix}
		2- \sin(1) & 1 + \cos(1) & 2 \\ 
		1+\cos(1) & -8 & -2 \\
		2 & -2 & 0 
		\end{pmatrix}.
		\]
		The choice of
		\[
		V = 
		\begin{pmatrix}
		-0.1968 & 0.9459 & 0.2578  \\
		0.5659  & 0.3243 & -0.7580 \\
		0.8006  & 0.0033 & 0.5992
		\end{pmatrix},\quad
		\sigma_1 = I\quad \mbox{and}\quad
		\sigma_2 = 
		\begin{pmatrix}
		-1  & 0 & 0\\
		0  & 1 & 0 \\
		0 &  0 & 1 
		\end{pmatrix},
		\]
		where $\sigma_1,\sigma_2 \in \Sigma_V(H(\bar x))$ leads to
		\[
		\gamma_1 = I\quad \mbox{and}\quad
		\gamma_2 = 
		\begin{pmatrix}
		0.9225 & 0.3723 & 0.1015\\
		0.3723  & -0.7896 & -0.4877 \\
		0.1015 &  -0.4877 & 0.8671 
		\end{pmatrix}.
		\]
		For $h=(0.2,0.05,0.1)^T$ we obtain
		\[
		f(\bar x+h) + f(\bar x-h) -  f(\bar x+\gamma_2 h) - f(\bar x-\gamma_2 h) \approx 6.40\cdot 10^{-5},
		\]
		and for $h=\frac{1}{10}(0.2,0.05,0.1)^T$ one computes
		\[
		f(\bar x+h) + f(\bar x-h) -  f(\bar x+\gamma_2 h) - f(\bar x-\gamma_2 h) \approx 6.38\cdot 10^{-9}
		\]
		as expected.
	\end{example}
	
	\section*{Acknowledgements}
	
	This work is supported by the Priority Programme SPP 1881 \emph{Turbulent Superstructures} of the Deutsche Forschungsgemeinschaft. We also thank an anonymous reviewer for important comments on the contents of this article.
	
	\bibliographystyle{unsrt}
	\bibliography{SCSN}

\end{document}